\RequirePackage{fix-cm}

\documentclass[smallextended]{svjour3}       
\usepackage{amsmath}
\smartqed  
\usepackage{arabtex}
\usepackage{graphicx,  stmaryrd}
\usepackage{color}
\usepackage{amssymb}
\usepackage{framed}
\usepackage{wrapfig}
\usepackage{hyperref}
\usepackage{adjustbox}
\hypersetup{
	colorlinks=true,
	linkcolor=black,
	urlcolor=red,
}

\newcommand{\F}{\mathsf F}
\renewcommand{\L}{\mathsf L}

\newcommand{\Q}{\mathbf Q}
\renewcommand{\R}{\mathbf R}
\newcommand{\Jimm}{{\adjustbox{scale=.8, raise=1ex, trim=0px 0px 0px 7px, padding=0ex 0ex 0ex 0ex}{\RL{j}}}}

\newcommand{\Z}{\mathbf Z}

\newcommand{\C}{\mathbf C}

\DeclareMathOperator{\cds}{cds}
\DeclareMathOperator{\con}{con}

\DeclareMathOperator{\num}{num}

\newcommand{\psl}{ \mathsf{PSL}_2 (\Z)    }

\newcommand{\pgl}{ \mathsf{PGL}_2 (\Z)   }

\definecolor{green}{RGB}{117, 165, 50}

\usepackage{float}
\usepackage{forest}
\usepackage{authblk}
\newcommand{\sherh}[1]{\fboxsep=0pt\setlength{\fboxrule}{1pt}
\begin{center}
   \fbox{\colorbox{green}{
         \begin{minipage}[t]{17cm}
            #1
         \end{minipage}
      }
   }
\end{center}}
\newcommand{\sherhh}[1]{\fboxsep=0pt\setlength{\fboxrule}{1pt}
\begin{center}
   \fbox{\colorbox{yellow}{
         \begin{minipage}[t]{17cm}
            #1
         \end{minipage}
      }
   }
\end{center}}

\newcommand{\sherhhh}[1]{\fboxsep=0pt\setlength{\fboxrule}{1pt}
\begin{center}
   \fbox{\colorbox{red}{
         \begin{minipage}[t]{17cm}
            #1
         \end{minipage}
      }
   }
\end{center}}

\renewcommand{\sherh}[1]{}\renewcommand{\sherhh}[1]{}
\renewcommand{\sherhhh}[1]{}

\begin{document}

\title{The Conumerator and the Codenominator
}

\titlerunning{The Conumerator and the Codenominator}        

   \author{ A. Muhammed Uluda\u{g}   \and   Buket Eren G\"{o}kmen  }

\authorrunning{B. Eren G\"{o}kmen \and  A. M. Uluda\u g} 

\institute{Buket Eren G\"{o}kmen  
	\email{bktern@hotmail.com} 
	\at Graz University of Technology, Institute of Analysis and Number Theory, Gratz, Austria
	\and
      A. Muhammed Uluda\u{g}   	\email{muhammed.uludag@gmail.com}  \at
Department of Mathematics, Galatasaray University, \.{I}stanbul, Turkey \\
}

\date{Received: date / Accepted: date}

\maketitle

\begin{abstract}
In this paper, we answer the question: ``{what is the $q$th Fibonacci number, where $q$ is a positive rational?}". The answer is the codenominator function, which is an integral-valued map. It is defined via a pair of functional equations. 
Many Fibonacci identities in the literature are valid for the codenominator. It is related to the recently introduced involution Jimm and to the outer automorphism of PGL(2,Z).
\end{abstract}

\section{Introduction}\label{intro}
Consider the \emph{numerator  function} $\num:\Q_{>0}\to \Z_{>0}$ defined as
\begin{equation*}
\num: x=(p,q)\in \Q_{>0}\to \num(x)=p \in \Z_{>0},
\end{equation*}
where $p, q>0,$ $ \gcd(p,q)=1$. It satisfies the functional equations
\begin{eqnarray}\label{numerator}
\num(1+x)=\num(x)+\num(1/x), \label{nmr111}\\
\num\left(\frac{x}{1+x}\right)=\num(x) \label{nmr222}
\end{eqnarray}
\sherhh{
$$
(f|T)=(f|I)+(f|U), \quad (f|UTU)=(f|I)
$$
One can consider these functional equations as equations for a map $\pgl\to V$, where $V$ is some vector space.
The map (which is well-defined up to a sign)
$$
f:\frac{px+q}{rx+s}\to q
$$
}
\noindent
subject to the condition $\num(1):=1$.
These determine the $\num$ function completely on $\Q_{>0}$, and it satisfies the additional equation $\num(x)=x\num(1/x)$.

\sherhh{Note that one has the equations
$$
\num(x-1)=\num(x)-\num(1/x),\\
\num\left(\frac{x}{x-1}\right)=\num(x) 
$$
and for now I don't know how to deduce these equations from the above ones. HW. do this? Do you have similar equations for the conumerator?
}
\par
The \emph{conumerator} is the function $\Q_{>0}\to \Z_{>0}$ 
defined by the system
\begin{align}
\con(1+x)&=\con(x)+\con(1/x), \label{cnm1}\\
\con\left(\frac{1}{1+x}\right)&=\con(x) \label{cnm2}
\end{align}
with $\con(1):=1$.  
\sherhh{
$$
(f|T)=(f|I)+(f|U), \quad (f|UT)=(f|I) \iff
$$
$$
(f|T)=(f|I)+(f|U), \quad (f|UTU)=(f|U) 
$$
}
Setting $y:=x+1>0$ we may rewrite this system as
\begin{eqnarray}\label{conumerator22}
\con(y)=\con(y-1)+\con\left(\frac{1}{y-1}\right), \quad
\con\left({1}/{y}\right)=\con(y-1).
\end{eqnarray}
If we express $y\in \Q_{>0}$ as a finite continued fraction 
$y=[n_0,n_1, \cdots, n_k]$ and set $\ell(y):=n_0+n_1+ \cdots+n_k$, then 
since both equations express $\con(y)$ and $\con(1/y)$  in terms of $\con$-values of rational numbers with 
$\ell$-value $\ell(y)-1<\ell(y)$, equations are consistent and $\con(y)$ can be computed in terms of $\con(1)$. 

The {\it codenominator} function $\F:\Q_{>0}\to \Z_{>0}$ is defined as $\F(x):=\con(1/x)$.  Thus, it is defined  by the system
\begin{align}
\F(1+1/x)&=\F(x), \label{cdnm1}\\
\F\left(\frac{1}{1+x}\right)&=\F(x)+\F(1/x) \label{cdnm2}
\end{align}
with $\F(1):=1$.  Computing $\F(x+2)$ by (\ref{cdnm1}-\ref{cdnm2}) we get 
\begin{align}\label{recurse}
\F(x+2)=\F(x+1)+\F(x),
\end{align}
so $\F$  extends the Fibonacci sequence to $\Q_{>0}$; i.e. $(\F(n))_{n\in \Z_{>0}}$ is the sequence defined by   
$
F_{n+1}=F_n+F_{n-1},
$
with $F_1=F_2=1$. 
 Since $\con(x)=\F(1/x)=\F(1+x)$, we see that $(\con(x))_{n\in \Z_{>0}}$ is the Fibonacci sequence shifted by one, which is preferred in combinatorial interpretations.

One can express the numerator in terms of $\F$ (Thm.~\ref{express1}). In \ref{codiscriminantt}   we introduce a 2-periodic function $\cds:\Q_{>0}\to \Z$. We then generalize some classical Fibonacci identities to $\F$ (Thm.~\ref{tag}).

In Cor.~\ref{prop2}, we relate $\F$ to the
recently discovered involution $\Jimm$ (``{jimm}") and show that $\F$ (as well as $\con$) takes every positive integral value infinitely often.
$\Jimm$ transforms non-noble real quadratic irrationals to real quadratic irrationals, 
simplifying considerably the Markov irrationals (\cite{eren}). 

Maple codes to compute the functions introduced here are available upon request.

\sherhh{An extension to $\Z^2$:
We may consider the following map:
$$
f:(m,n)\in \Z_{>0}^2 \to \gcd(m,n)(\con(m/n), \con(n/m))\in \Z_{>0}^2 
$$
This gives another way to enumerate the lattice. One has
$$
f(m+n,n)=\gcd(m+n,n)(\con\frac{m+n}{n}, \con\frac{n}{m+n}), 
=\gcd(m,n)(\con(1+\frac{m}{n}), \con\frac{1}{1+m/n}), 
$$
$$
=\gcd(m,n)(\con\frac{m}{n}+\con\frac{n}{m}, \con\frac{m}{n})=f(n,m)+\gcd(m,n)(\con\frac{m}{n},0)
$$
$$
f(m,m+n)=\gcd(m,n+m)(\con\frac{m}{m+n}, \con\frac{m+n}{m}),
=\gcd(m,n)(\con\frac{1}{1+n/m}, \con(1+\frac{n}{m}) ), 
$$
$$
=\gcd(m,n)( \con\frac{n}{m},\con\frac{m}{n}+\con\frac{n}{m})=f(n,m)+\gcd(m,n)(\con\frac{n}{m},0)
$$
$$
\implies
f(m+n,n)+f(m,m+n)=2f(n,m)+f(m,n)
$$
}
\section{Basic properties of the codenominator}
Set $F_0:=0$ and $F_{-n}:=(-1)^{n+1}F_n$ for $n>0$, where $F_n$ is the $n$th Fibonacci number.
\begin{proposition}\label{fes}
(i) For $x\in (0,1)\cap \Q$, one has $\F(x)=\F(1-x)$. \\
(ii) For $n\in \Z$ and $x\in \Q_{>0}$ such that $n+x>0$
\begin{eqnarray}\label{conums}
\F(n+x)=F_{n}\F(1+x)+F_{n-1}\F(x).\label{conums2}
\end{eqnarray}
(iii) For $n\in \Z_{>0}$ and $x\in \Q_{>0}$
\begin{eqnarray}\label{conumss}
\F\left(\frac{F_{n}+F_{n+1}x}{F_{n-1}+F_{n}x}\right)=\F(x)
\label{conums3}
\end{eqnarray}
\end{proposition}
\begin{proof} 
(i) The right hand side of (\ref{cdnm1}) is invariant under $x\leftrightarrow 1/x$. As for the left hand side 
$$
\frac{1}{1+x}+\frac{1}{1+1/x}=1
\implies 
\F\left(\frac{1}{1+x}\right)=\F\left(\frac{x}{1+x}\right)
$$
(ii) The cases $n=0,1$ are easily checked and the case $n=2$ reduces to (\ref{recurse}).
Now suppose (\ref{conums2}) is true for $n$. Then use (\ref{cdnm1}) and (\ref{cdnm2}) to get
\begin{align*}
\F(n+1+x)&=F_{n}\F(x+2)+F_{n-1}\F(x+1)\\
&=F_{n+1}\F(x+1)+F_{n}\F(x),
\end{align*}
so (\ref{conums2}) is true for $n+1$ as well. This settles the case $n\geq 0$, and the case $n<0$ is proved similarly.\\
(iii) This is just an iteration of (\ref{cdnm1}). 
\qed\end{proof} 
By Proposition \ref{fes}, if $x=[n_0,n_1,\dots,n_k]$ is given as a continued fraction, then:
\begin{align*}
\F([n_0,\dots,n_k]) &=F_{n_0}\F([1,n_1,\dots,n_k])+F_{n_0-1}\F([0,n_1,\dots,n_k]) \\
&=F_{n_0}\F([n_1,n_2,\dots,n_k])+F_{n_0-1}\F([n_1+1,\dots,n_k])
\implies \\
\F([n_0,n_1])&=
F_{n_0}F_{n_1}+F_{n_0-1}F_{n_1+1}\\
\F([n_0,n_1,n_2])&=
F_{n_0}(F_{n_1}F_{n_2}+F_{n_1-1}F_{n_2+1}
)+F_{n_0-1}(F_{n_1+1}F_{n_2}+F_{n_1}F_{n_2+1}).
\end{align*}
Invoking Binet's formula, one can express $\F([n_0, n_1,\dots, n_k])$ in terms of the golden ratio $\varphi:=(1+\sqrt{5})/2$, for example
\begin{align*}
\F([n_0,n_1])
=\frac{\varphi^{n_0}-\bar\varphi^{n_0}}{\varphi-\bar\varphi}
\frac{\varphi^{n_1}-\bar\varphi^{n_1}}{\varphi-\bar\varphi}
+
\frac{\varphi^{n_0-1}-\bar\varphi^{n_0-1}}{\varphi-\bar\varphi}
\frac{\varphi^{n_1+1}-\bar\varphi^{n_1+1}}{\varphi-\bar\varphi}.
\end{align*}
In particular for $n=1,2,\dots$
$$
\F(n+1/2)=F_nF_2+F_{n-1}F_3=F_{n-1}+F_{n+1} \mbox{ (OEIS:A000204)}
$$ is the Lucas sequence. Since $\F({x+1})=\F(1/x)$ and $\F({x-1})=\F(1/x)-\F(x)$, we define the Lucas function on $\Q_{>0}$ as
$
\L(x):=2\F(1/x)-\F(x)
$.

\sherhh{
\begin{align*}
=F_{n_0+1}(F_{n_1}\con([n_2-1,n_3,\dots,n_k])+F_{n_1-1}\con([n_2,\dots,n_k]))\\
+F_{n_0}(F_{n_1+1}\con([n_2-1,n_3,\dots,n_k])+F_{n_1}\con([n_2,\dots,n_k]))
\end{align*}
\begin{align*}
=(F_{n_0+1}F_{n_1}+F_{n_0}F_{n_1+1})\con([n_2-1,n_3,\dots,n_k])+
(F_{n_0+1}F_{n_1-1}+F_{n_0}F_{n_1}) \con([n_2,\dots,n_k]))
\end{align*}
\begin{align*}
=(F_{n_0+1}F_{n_1}+F_{n_0}F_{n_1+1})
(F_{n_2}\con([n_3-1,n_2,\dots,n_k])+F_{n_2-1}\con([n_3,\dots,n_k]))+\\
(F_{n_0+1}F_{n_1-1}+F_{n_0}F_{n_1}) 
(F_{n_2+1}\con([n_3-1,n_2,\dots,n_k])+F_{n_2}\con([n_3,\dots,n_k]))
\end{align*}
\begin{align*}
=(F_{n_2}(F_{n_0+1}F_{n_1}+F_{n_0}F_{n_1+1})+F_{n_2+1}(F_{n_0+1}F_{n_1-1}+F_{n_0}F_{n_1}))
\con([n_3-1,n_2,\dots,n_k]+\\
(F_{n_2-1}(F_{n_0+1}F_{n_1}+F_{n_0}F_{n_1+1})+F_{n_2}(F_{n_0+1}F_{n_1-1}+F_{n_0}F_{n_1}))
\con([n_3,\dots,n_k]))
\end{align*}}
\sherhh{
And there is a Binet formula, involving sums of $\pm$ of partial fractions...
$$
\con([n_0,\dots,n_k]) =
\frac{\varphi^{n_0+1}-\bar\varphi^{n_0+1}}{\varphi-\bar\varphi}\con([n_1-1,n_2,\dots,n_k])+
\frac{\varphi^{n_0}-\bar\varphi^{n_0}}{\varphi-\bar\varphi}\con([n_1,\dots,n_k])=
$$
\begin{align*}
\frac{\varphi^{n_0+1}-\bar\varphi^{n_0+1}}{\varphi-\bar\varphi}
\left(
\frac{\varphi^{n_1}-\bar\varphi^{n_1}}{\varphi-\bar\varphi}\con([n_2-1,n_3,\dots,n_k])+
\frac{\varphi^{n_1-1}-\bar\varphi^{n_1-1}}{\varphi-\bar\varphi}\con([n_2,\dots,n_k])
\right)+\\
\frac{\varphi^{n_0}-\bar\varphi^{n_0}}{\varphi-\bar\varphi}
\left(
\frac{\varphi^{n_1+1}-\bar\varphi^{n_1+1}}{\varphi-\bar\varphi}\con([n_2-1,n_3,\dots,n_k])+
\frac{\varphi^{n_1}-\bar\varphi^{n_1}}{\varphi-\bar\varphi}\con([n_2,\dots,n_k])
\right)
\end{align*}
something like this:
\begin{align*}
\con([n_0,n_1])=F_{n_0+1}F_{n_1}+F_{n_0}F_{n_1+1}
=\frac{\varphi^{n_0+1}-\bar\varphi^{n_0+1}}{\varphi-\bar\varphi}
\frac{\varphi^{n_1}-\bar\varphi^{n_1}}{\varphi-\bar\varphi}
+
\frac{\varphi^{n_0}-\bar\varphi^{n_0}}{\varphi-\bar\varphi}
\frac{\varphi^{n_1+1}-\bar\varphi^{n_1+1}}{\varphi-\bar\varphi}
\end{align*}
\begin{align*}
({\varphi-\bar\varphi})^2\con([n_0,n_1])=\\
\varphi^{n_0+n_1+1}-
\varphi^{-n_0+n_1-1}-
\varphi^{n_0-n_1+1}+
\varphi^{-n_0-n_1-1}+
\varphi^{n_0+n_1+1}-
\varphi^{-n_0+n_1+1}-
\varphi^{n_0-n_1-1}+
\varphi^{-n_0-n_1-1}
\end{align*}
\begin{align*}
=
2\varphi^{n_0+n_1+1}-
\varphi^{-n_0+n_1-1}-
\varphi^{n_0-n_1+1}-
\varphi^{-n_0+n_1+1}-
\varphi^{n_0-n_1-1}+
2\varphi^{-n_0-n_1-1}
\end{align*}
}
Setting $x=1$ in (\ref{conums3}) and in (\ref{conums2}) we get 
\begin{corollary}\label{cor1} For $n\in \Z_{>0}$ 
\begin{eqnarray*}
\F\left(\frac{F_n}{F_{n+1}}\right)=n, \quad
\F\left(\frac{F_{n+1}}{F_{n}}\right)=1, \quad
\F\left(\frac1n\right)=\F\left(1+n\right)=F_{n+1}.
\end{eqnarray*}
\end{corollary}
One has
\begin{eqnarray*}
\F\left(n+{1}/{3}\right)=2F_{n+1}+F_{n-1} \mbox{ (OEIS:A001060)}\\
\F\left(n+{1}/{4}\right)=3F_{n+1}+2F_{n-1} \mbox{ (OEIS:A022121)}\\
\F\left(n+{1}/{5}\right)=5F_{n+1}+3F_{n-1} \mbox{ (OEIS:A022138)}\\ 
\F\left(n+{1}/{n}\right)=2F_{n}^2+(-1)^n \mbox{ (OEIS:A061646)}\\
\F([n,n,n])
=F_n( 4F_{n}^2+3(-1)^n)
\end{eqnarray*}

Finally, let us give some sporadic computations:
\begin{eqnarray*}
\F\left({23}/{27}\right)=97 \mbox{ (prime)},\quad 
\F\left({27}/{23}\right)=37 \mbox{ (prime) }\\
\F\left({19}/{41}\right)=2^4\times 7\quad 
\F\left({41}/{19}\right)=3^4
\end{eqnarray*}
See Appendix I for the list of values $\F(41/k)$,  $k=1,2,..200$.

\medskip
\noindent
{\bf Problem.} Characterize those rationals with a prime codenominator.

\begin{theorem} \label{express1}
The numerator is expressed in terms of the codenominator as:
\begin{equation}\label{express}
\num(x)=\F\left(\frac{\F(x)}{\F(1/x)}\right)
\end{equation}
\end{theorem}

\begin{proof} 
Denote the right hand side of (\ref{express}) by $f(x)$. Since $f(1)=1=\num(1)$, it suffices to prove that $f$ satisfies~(\ref{nmr111}-\ref{nmr222}). It satisfies~(\ref{nmr111}) since
\begin{eqnarray*}
f(x+1)=\F\left(\frac{\F(x+1)}{\F(1/(x+1))}\right)
=\F\left(\frac{\F(1/x)}{\F(x)+\F(1/x)}\right)\\
=\F\left(\frac{1}{1+\frac{\F(x)}{\F(1/x)}}\right)
=\F\left(\frac{\F(x)}{\F(1/x)}\right)+
\F\left(\frac{\F(1/x)}{\F(x)}\right)=
f(1/x)+f(x),
\end{eqnarray*}
As for~(\ref{nmr222}),
\begin{eqnarray*}
f\left(\frac{x}{x+1}\right)=\F\left(\frac{\F(\frac{x}{x+1})}{\F(\frac{1}{\frac{x}{x+1}})}\right)=
\F\left(\frac{\F(\frac{1}{1+1/x})}{\F({{1+1/x}})}\right)=\\
\quad\quad\quad\quad\F\left(\frac{\F(x)+\F(1/x)}{\F(x)}\right)=
\F\left( \frac{\F(x)}{\F(1/x)} \right)
=f(x).\quad\quad\quad\quad\qed
\end{eqnarray*}
\end{proof}
\begin{corollary}\label{crrr}
The codenominator satisfies
\begin{eqnarray*}
\F
\left(\frac{\F(x)}{\F({1}/{x})}\right)=
x\F\left(\frac{\F({1}/{x})}{\F(x)}\right)\label{involut0}
\end{eqnarray*}
\end{corollary}

\subsection{The codiscriminant}\label{codiscriminantt}
The {\it codiscriminant} function is defined for $x\in \Q_{>0}$ as
\begin{align}
\cds(x)&:=\F(1/x)^2-\F(x)\F(1/x)-\F(x)^2\label{cdss} \\
&=\left(\F(1/x)- \varphi \F(x)\right)\left(\F(1/x)-\bar\varphi \F(x)\right)\nonumber\\
&=\F(x+1)^2-\F(x)\F(x+1)-\F(x)^2\nonumber\\
&=\F(x+1)^2-\F(x)\F(x+2)\nonumber
\end{align}
\begin{lemma}\label{transinv}
For $x\in \Q_{>0}$ one has $\cds(1+x)=-\cds(x)$. Hence, $\cds$ is 2-periodic on $\Q_{>0}$.
\end{lemma}
\begin{proof} Use (\ref{cdss}) to get 
\begin{align*}
\cds(1+x)
&=
\left(\F(1/(1+x))- \varphi \F(1+x)\right)\left(\F(1/(1+x))-\bar\varphi \F(1+x)\right)\\
&=
\left(\F(x)+\F(1/x)- \varphi \F(1/x)\right)\left(\F(x)+\F(1/x)-\bar\varphi \F(1/x)\right)\\
&=
\left(\F(x)+(1-\varphi) \F(1/x)\right)\left(\F(x)+(1-\bar\varphi) \F(1/x)\right)
\end{align*}
Now use $1-\varphi=-1/\varphi$, $1-\bar \varphi=-1/\bar\varphi$ and $\varphi\bar\varphi=-1$.\qed
\end{proof}
For $x=n\in \Z_{>0}$ this reduces to the Cassini identity
$$
\cds(n)=F_{n+1}^2-F_{n+1}F_{n}-F_{n}^2=(-1)^n.
$$
Similarly, 
$
\cds\left(n+1/k\right)
 =(-1)^n\cds\left(1/k\right) \nonumber
$
gives the codenominator identity
\begin{eqnarray*}
(F_{n+1}F_k+F_{k+1}F_n)^2-(F_{n+1}F_k+F_{k+1}F_n)(F_{n}F_k+F_{k+1}F_{n-1})
\\-(F_{n}F_k+F_{k+1}F_{n-1})^2=(-1)^{n}\cds(1/k).
\end{eqnarray*}
Here, $(-\cds(1/k)=\F(1+k)\F(2+k)-\F(k)^2)$ is the sequence
$$
1, 1, 5, 11, 31, 79, 209, 545, 1429, 3739, 9791 \dots \mbox{ (OEIS:A236428)}.
$$

\begin{lemma}\label{transinv2}
For $x\in (0,1)\cap \Q$ one has $\cds(1-x)=\cds(x)$. Hence,
$\cds(n-x)=(-1)^{n+1}\cds(x)$ for $n>x$, $n\in \Z$.
\end{lemma}
\begin{proof}
Note that $(x/(1+x))+(1/(1+x))=1$.
So we check $\cds(x/(1+x))=\cds(1/(1+x))$ for all $x\in \Q_{>0}$:
\begin{eqnarray*}
\cds\left(\frac{x}{1+x}\right) = 
\F^2\left(\frac{1+x}{x}\right)-
\F\left(\frac{x}{1+x}\right) \F\left(\frac{1+x}{x}\right)
- \F^2\left(\frac{x}{1+x}\right) \\
=
\F^2(x)-(\F(x)+\F(1/x))\F(x)-(\F(x)+\F(1/x))^2\\
=
-\F^2(1/x)-3\F(x)\F(1/x)-\F^2(x)
\end{eqnarray*}
\begin{eqnarray*}
\cds\left(\frac{1}{1+x}\right) = 
\F^2\left({1+x}\right)-
\F\left(1+x\right) \F\left(\frac{1}{1+x}\right)
- \F^2\left(\frac{1}{1+x}\right) \\
=\F^2(1/x)-\F(1/x)(\F(x)+\F(1/x))-(\F(x)+\F(1/x))^2\\
=\quad\quad\quad\quad-\F^2(1/x)-3\F(x)\F(1/x)-\F^2(x)\quad\quad\qed
\end{eqnarray*}
Iterating the proof of the lemma gives
\begin{eqnarray*}
\cds\left(\frac{F_{n-1}x+F_n}{F_nx+F_{n+1}}\right)=
-\bigl((n^2+n-1)\F^2(1/x)+(2n+1)\F(1/x)\F(x)+\F(x)^2\bigr)\\
\end{eqnarray*}
\end{proof}
%
%
The function $\cds$ is therefore invariant up to sign under the transformations 
$n+x$ and $n-x$ forming a sort of dihedral semigroup, with 
the interval $(0,1/2]$ as its fundamental region. It is not injective on this region; e.g.
$\cds(4/15)=\cds(29/76)=-319$.

\section{Fibonacci identities}
Among the myriad Fibonacci identities in the literature~\cite{koshy}, many admit a codenominator interpretation. The idea is to replace an integer $n$ by a rational $q$ and $(-1)^n$ by $\cds(q)$ in the formula. Here are some examples.
\begin{theorem}\label{tag}
If at least two among $p,q,r \in \Q_{>0}$ are integral, then
\begin{eqnarray}\label{taguiri}
\F({p+q})\F({p+r})-\F(p)\F({p+q+r})=
\cds(p)\F({q})\F(r)
\end{eqnarray}
(reduces to Taguiri's identity when $p, q, r \in \Z$).
\end{theorem}
\begin{proof}
If $p,q\in\Z_{>0}$, then 
\begin{eqnarray*}
\F({p+q})\F({p+r})-\F(p)\F({p+q+r})\\
=F_{p+q}(F_p\F(r+1)+F_{p-1}\F(r))-
F_p(F_{p+q}\F(r+1)+F_{p+q-1}\F(r))\\
=(F_{p-1}F_{p+q}-F_pF_{p+q-1})\F(r))\\
(*) =(-1)^{p+q-1}F_{-q}\F(r)=(-1)^{p+q-1}(-1)^{q+1}F_q\F(r)\\
=\cds(p)\F(q)\F(r)
\end{eqnarray*}
where in step $(*)$ we used d'Ocagne's identity: 
$F_sF_{n+1}-F_{s+1}F_n=(-1)^nF_{s-n}$ valid for $s,n\in \Z$.

If $q,r\in\Z_{>0}$, then 
{\small \begin{eqnarray*}
\F({p+q})\F({p+r})-\F(p)\F({p+q+r})=\\
(F_q\F(1/p)+F_{q-1}\F(p))
(F_r\F(1/p)+F_{r-1}\F(p))
-\F(p)(F_{q+r}\F(1/p)+F_{q+r-1}\F(p))\\
=F_qF_r\F^2(1/p)
+(F_qF_{r-1}+F_rF_{q-1}-F_{q+r})\F(p)\F(1/p)+
(F_{q-1}F_{r-1}-F_{q+r-1})\F(p)^2.
\end{eqnarray*}
}
To finish the proof, use the Fibonacci identities 
$$
F_{q-1}F_{r-1}-F_{q+r-1}=F_qF_{r-1}+F_rF_{q-1}-F_{q+r}=-F_qF_r
\qquad \qed
$$
\end{proof}
Setting $n=q=r$ to be integral and $p+q=p+r=s$ in (\ref{taguiri}) gives
\begin{corollary}
For $s\in \Q_{\geq0}$, $n\in \Z_{\geq 0}$, $s\geq n$ one has 
$
\F(s)^2-\F({s+n})\F({s-n})=\cds(s+n)\F^2(n)
$
(reduces to Catalan's identity when $s\in \Z$).
\end{corollary}
Setting $r=1$, $n=p$, $s=p+q$ in (\ref{taguiri}) gives
\begin{corollary}
For $s\in \Q_{\geq0}$, $n\in \Z_{>0}$, $s\geq n$ one has
$
\F(s)\F({n+1})-\F({s+1})\F(n)=\cds(n)\F({s-n})
$
(reduces to d'Ocagne's identity when $s\in \Z$)
\end{corollary}
Here is a Lucas identity which follows directly from definitions:
\begin{proposition}
For $q\in \Q_{>0}$ one has 
$\L(q)^2-5\F(q)^2=4\cds(q)$.
\end{proposition}
Not every Fibonacci identity has a straightforward codenominator version: e.g. we failed\footnote{On the other hand, the formula $\F(n+1)\F(q+1)-\F(n-1)\F(q-1)=\F(q+n)$ is valid.} to interpret  
$F_{n+1}^2-F_{n-1}^2=F_{2n}$. Experiments suggest that 
$(\F^2(n+q+1)-\F^2(n+q-1)/\F(2n+2q)$ tends to a limit in 
$\Q(\sqrt{5})$ as $n\to\infty$.

%
%

\sherhh{

Generating Functions

Let $q\in (0,1]$. Set
$$
s_q(x):=\sum_{n=0}^\infty \con(q+n)x^n
$$
Then
$$
s_q(x)=\con(q)+\con(q+1)+\sum_{n=2}^\infty (\con(q+n-1)+\con(q+n-2))x^n
$$
$$
=\con(q)+\con(q+1)+\sum_{n=2}^\infty \con(q+n-1)x^n+\sum_{n=2}^\infty \con(q+n-2)x^n
$$
$$
=\con(q)+\con(q+1)x+x\sum_{n=1}^\infty \con(q+n)x^n+x^2\sum_{n=0}^\infty \con(q+n)x^n
$$
$$
=\con(q)+\con(q+1)x+x(s_q(x)-\con(q)) +x^2 s_q(x)
$$
$$
\implies s_q(x)=\frac{\con(q)+x(\con(q+1)-\con(q))}{1-x-x^2}  =\frac{\con(q)+x\con(\frac1q)}{1-x-x^2} 
$$

Another try. Set (meaninglessly)

$$
s(x):=\sum_{q\in \Q_{>0}} \con(q)x^q\implies
$$
$$
s(x)=\sum_{q\in (0,1]} \sum_{n=0}^\infty \con(q+n)x^{q+n}
$$
$$
=\sum_{q\in (0,1]} \left(  \con(q)x^q+\con(q+1)x^{q+1}+ \sum_{n=2}^\infty \con(q+n)x^{q+n} \right)
$$
$$
=\sum_{q\in (0,1]} \left(  \con(\frac1q-1)x^q+\con(q)x^{q+1}+\con(\frac1q)x^{q+1}+ 
\sum_{n=2}^\infty (\con(q+n-1)+\con(q+n-2))x^{q+n} \right)
$$
One has
$$
\sum_{q\in (0,1]} \sum_{n=2}^\infty (\con(q+n-1)+\con(q+n-2))x^{q+n} =
xs(x)+x^2s(x)\implies
$$
$$
s(x)=xs(x)+x^2s(x)+
\sum_{q\in (0,1]}  \con(\frac1q-1)x^q+
\sum_{q\in (0,1]} \con(q)x^{q+1}+
\sum_{q\in (0,1]} \con(\frac1q)x^{q+1}
$$
$$
=xs(x)+x^2s(x)+
\sum_{q\in \Q_{>0}}  \con(q)x^{\frac{1}{q+1}}+
x\sum_{q\in (0,1]} \con(q)x^{q}+
\sum_{q\in [1,\infty]} \con(q)x^{\frac1q+1}
$$
$$
=xs(x)+x^2s(x)+
\sum_{q\in \Q_{>0}}  \con(q)x^{\frac{1}{q+1}}+
x\sum_{q\in (0,1]} \con(q)x^{q}+
x\sum_{q\in [1,\infty]} \con(q)x^{\frac1q}
$$

Of course, you can not expect to have anything algebraic for $s(x)$ with this definition..

One might try a generating function with infinitely many variables, something like
$$
s(x)=\sum_{q\in \Q_{>0}} \con(q)x_0^{n_0}x_1^{n_1}\dots x_k^{n_k}
$$
where $q=[n_0,n_1,\dots, n_k]$. One may try this for the function $\num$ as well.
One might start with the functions
$$
s_0(x_0):=\sum \con([n_0])x_0^{n_0}, \quad s_1(x_0,x_1):=\sum \con([n_0,n_1])x_0^{n_0} x_1^{n_1} ,\quad etc
$$

%
}

\section{The involution Jimm}
\begin{definition}
The  function below is called Jimm:
$$ 
\Jimm: x\in \Q_{>0}\to \frac{\con(x)}{ \con(1/x)} 
=\frac{\F(1/x)}{\F(x)}\in \Q_{>0}
$$ 
\end{definition}

\begin{lemma}\label{lemma2} For all $x \in \Q_{>0}$ one has
	\begin{enumerate}
		\item [i.] $\Jimm(\Jimm(x))=x $  
		\item [ii.] \label{ii} $ \Jimm(1/x)=1/\Jimm(x)$
		\item [iii.]\label{iii} $ \Jimm(1+x)=1+1/\Jimm(x)$ $\iff$ $\Jimm(1+1/x)=1+\Jimm(x)$ 
		\item [iv.] $\Jimm(1-x)=1-\Jimm(x)$ $(0<x<1)$.
	\end{enumerate} 
\end{lemma}


\begin{proof} 
		i) Use (\ref{express}) to get:\\[-6mm]
		\begin{align*}
		\Jimm(\Jimm(x))=\frac{{ \F\Bigg(\dfrac{ \F(x)}{\F(1/x)}}\Bigg)}{{ \F\Bigg(\dfrac{ \F(1/x)}{\F(x)}}\Bigg)}
		=\frac{\num(x)}{\num(1/x)}=x.
		\end{align*}
		ii)
		$\Jimm(1/x)=\con(1/x) / \con(x)=1/ (\con(x) / \con(1/x))=1/\Jimm(x)$.\\
		iii)		
		$\Jimm(1+x)=\con(1+x)/\con(1/(1+x))=(\con(x)+\con(1/x))/\con(x)$ 
		$=1+\con(1/x)/\con(x)$
		=$1+1/\Jimm(x)$. \\
		iv) 
		Set $t:=1/x-1$. Then $x=1/(1+t)$ and $1-x=t/(1+t)=1/(1+1/t)$. By (ii) and (iii) we get 
		$\Jimm(x)= 1/(1+1/\Jimm(t))=\Jimm(t)/(1+\Jimm(t))$ and $\Jimm(1-x)=1/(1+\Jimm(t))$, from which it follows that 
		$\Jimm(x)+\Jimm(1-x)=1$. \qed
		\end{proof}

	Recall that every $x\in \Q\setminus\{0,1\}$ has two simple continued fraction representations, one ending with a $1$ and other not. The result below is independent of which representation is used for $x$.
\begin{corollary}\label{corrx}
	Let $x=[n_0,n_1, \dots, n_k ]>1$ be the simple continued fraction expansion of $x\in \Q_{>0}$. Denote by  $1_{k}$ the sequence $1,1,\dots , 1$ of length $k$.
	Then
		$$ 
		\Jimm(x)= [1_{n_0-1},2, 1_{n_1-2},2,1_{n_2-2},2, 
		\dots 
		2,1_{n_{k-1}-2},2,1_{n_k-1}]
		 $$
		with the rules: $
	[ \dots , n, 1_0, m, \dots ]:= [ \dots, n,m, \dots ]$ and
	$[ \dots , n, 1_{-1}, m, \dots ]:=  [ \dots , n+m-1, \dots ].
	$ 
\end{corollary}
\begin{proof} Lemma~\ref{lemma2}-(ii) and (iii) gives $\Jimm(1+1/x)=1+\Jimm(x)$. Apply this and 
Lemma~\ref{lemma2}-(iii) repeatedly to get
\begin{eqnarray*}
\Jimm([n_0,n_1,\dots,n_k])=\Jimm(1+[n_0 -1,n_1, \dots,n_k ] ) \\ 
=[1,\Jimm([n_0 -1,n_1, \dots,n_k ] )] 
=[1_2, \Jimm(1+[n_0-3,n_1, \dots ,n_k ])]\\
=[1_{n_0-1},\Jimm(1+[0,n_1, \dots ,n_k ])]
=[1_{n_0-1},\Jimm(1+1/[n_1,\dots ,n_k ])]\\
=[1_{n_0-1},1+\Jimm([n_1, \dots ,n_k ])]
=[1_{n_0-1},2,\Jimm([n_1-1,n_2, \dots ,n_k ])]\dots \quad\quad \qed
\end{eqnarray*}
\end{proof}

One can apply this formula for $x=[n_0,n_1, \dots, n_k ]$ with $0<x<1$ (i.e. $n_0=0$) with the additional rule: a $1_{-1}$ at the beginning of a sequence is replaced by $0$ (and a $1_{0}$ at the beginning of a sequence is ignored).

\subsection{Fibers of the conumerator}
By Equation~\ref{cnm2}, $\con$ takes on every value and infinitely often.
This also shows that its fibers splits into $1/(1+x)$-orbits, i.e.
$
\con(x)=\con([0,1,1,\dots,1,x]).
$
\begin{lemma}
(i) The fibers of the conumerator above $p\in \Z_{>0}$ are
\begin{equation*}
\con^{-1}(p)=\Bigl\{\Jimm\left({p}/{q}\right) \, | \, q\in \Z_{>0},\, (p,q)=1\Bigr\}
\end{equation*}
(ii) The number of $1/(1+x)$-orbits inside $\con^{-1}(p)$ is the totient $\bar\varphi(p)$.
\end{lemma}
\begin{proof}
\begin{itemize}
\item [i.] 
Since $\con(x)=\num(\Jimm(x))$, if $\con(x)=p$, then $\Jimm(x)=p/q$ for some $q$ coprime to $p$.
Since $\Jimm$ is involutive, we get
\begin{equation*}
\Jimm(x)=p/q \iff x=\Jimm(\Jimm(x))=\Jimm\left({p}/{q}\right).
\end{equation*}
\item[ii.] Suppose $x=\Jimm(p/q)$ with $(p,q)=1$, so that $\con(x)=p$.
Use Lemma~\ref{lemma2} to get 
\begin{equation*}
\frac{1}{1+x}=\frac{1}{1+\Jimm(p/q)}=\frac{1}{1+\frac{1}{\Jimm(q/p)}}=\frac{1}{\Jimm(1+q/p)}=\Jimm\left(\frac{p}{p+q}\right),
\end{equation*}
i.e. the number ${1}/{(1+x)}$-orbits in the set $\con ^{-1} (p)$  is the number of integers $q$ coprime to $p$ and with $q<p$, which is exactly  the totient  $ \bar\varphi(p)$.\qed
\end{itemize}
\end{proof}

\subsection{Extending Jimm, $\num$ and $\con$ to $\Q^*$}
We extend $\Jimm$ from $\Q_{>0}$ to $\Q^*:=\Q\setminus\{0\}$ by declaring $\Jimm(x):=-1/\Jimm(-x)$ for $x<0$ and keep denoting the extension by the same letter. 
The functional equations $ \Jimm(\Jimm(x))=x$, $\Jimm(1/x)=1/\Jimm(x)$ and 
$\Jimm(-x):=-1/\Jimm(x)$ are then obviously satisfied. On the other and, the extended function does not (always) satisfy 
$\Jimm(1-x)=1-\Jimm(x)$, since setting $x=1$ in this equation forces  $\Jimm(0):0=0$. On the other hand setting $x=0$ in $\Jimm(-x):=-1/\Jimm(x)$ forces $\Jimm(0)^2=-1$.

In a similar vein, note that it is natural to extend the numerator to $\Q^*$ by defining $\num(-p/q):=-\num(p/q)=-p$ for $p/q\in \Q_{>0}$; however, (\ref{nmr111}-\ref{nmr222}) will not hold for the extended function. Similarly, it is natural to define $\con(-p/q):=-\con(q/p)$ for $p/q\in \Q_{>0}$; with the cost that  (\ref{cnm1}-\ref{cnm2}) 
will not hold for the extended function.

It follows that the map $(p,q)\in \Z_{>0}^2 \to \gcd(p,q)(\F(p/q), \F(q/p))\in \Z_{>0}^2$ is a bijection, i.e. it gives an alternative indexing of the first quadrant of $\Z^2$. For example; for $\Re(s)>1$ one has
$$
\sum_{q\in \Q_{>0}} \frac{1}{\F(x)^s\F(1/x)^s}=\frac{\zeta(s)^2}{\zeta(2s)},
$$
where $\zeta$ is Riemann's zeta function.

\subsection{Extending Jimm to $\R\setminus\{0\}$}
The theorem below is taken from~\cite{Uludag} (see also ~\cite{UA}):
\begin{theorem}\cite{Uludag}
For every $y\in \R\setminus \Q$, the limit 
\begin{align*}
\lim_{x\in \Q^*, x\to y} \Jimm(x),
\end{align*}
exists, and if we extend $\Jimm$ to $\R$ by declaring $\Jimm(y)$ of $y\in \R\setminus \Q$ to be this limit, then the limit function 
$\Jimm: \R\setminus\{0\}\to \R\cup \{\infty\}$ is continuous on $\R\setminus \Q$ and discontinuous on $\Q$. Jimm satisfies the functional equations (ii)-(iv) of Lemma~\ref{lemma2} for $y\in \R\setminus \Q$ and it is involutive on a subset of $\R\setminus \Q$.
\end{theorem}
Consequently, the image of a positive irrational number  is given by 
\begin{align*}
\Jimm([n_0,n_1,n_2,n_3, \dots])=[{1}_{n_0-1},2, {1}_{n_1-2},2,{1}_{n_2-2},2, {1}_{n_3-2}\dots ]
\end{align*}
\begin{example}\label{exx}
	\textup{	Let us calculate the values of some periodic continued fractions :}
	\begin{enumerate}
		\item  $ \Jimm([1,1, \dots ])= \infty$. More generally, if $n_k>1$ then
		$\Jimm([n_0, n_1, \dots, n_k, 1,1\dots])=\Jimm([n_0, n_1, \dots, n_k]) \in \Q$.
		\item $\Jimm([2,2,2, \dots ]))=[1,2,2,2, \dots ]$ 
		\item $ \Jimm([n,n,n, \dots ])=[1_{n-1}, \overline{2,1_{n-2}}]$ 
		\item $ \Jimm([\overline{2,2,1,1}])=[1,\overline{2,4}]$
		\item $\Jimm([\overline{2,2,1,1,1,1}])=[1,\overline{2,6}]$
		\item $\Jimm([\overline{2,2,2,2,1,1}])=[1,\overline{2,2,2,4}]$
	\end{enumerate}
\end{example}
Observe in Example~\ref{exx}.1 that $\Jimm$ sends continued fractions ending with an infinite sequence of $1$'s (i.e. noble numbers) to $\Q$. 
An infinite continued fraction is eventually periodic if and only if  it represents a quadratic irrational.  
Since Jimm clearly preserves the set of non-noble irrationals having eventually periodic continued fraction expansions, we have
\begin{corollary} \label{prop2}\cite{Uludag}
	Jimm preserves the set of non-noble quadratic irrationals.
\end{corollary}
In fact, $\Jimm$ commutes with the Galois action on the set of real quadratic irrationals. It conjecturally sends algebraic numbers of degree $>2$ to transcendental numbers~\cite{trans}.
For some other properties of Jimm, see~\cite{UA2} and~\cite{UA3}.

Finally, note that, for $x\in \Q_{>0}$ one has
\begin{align*}
\frac{\cds(x)}{\F(x)^2}=\frac{\F(1/x)^2-\F(x)\F(1/x)-\F(1/x)^2}{\F(x)^2}=\Jimm(x)^2-\Jimm(x)-1,\\
\cds\left(\Jimm\left(p/q\right)\right)=p^2-pq-q^2 \quad (p,q\in \Z_{>0}).
\end{align*}
It is enlightening to consider the $\Jimm$-transforms of Fibonacci identities. For example, consider the identity~\cite{koshy} known to be valid for $q\in \Z$:
$$
\F(q+3)^2-2\F(q+2)^2-2\F(q+1)^2=\F(q)^2.
$$
To prove this for $q\in\Q_{>0}$, substitute $\Jimm(q)$ in place of $q$ to get
$$
\F(\Jimm(r)+3)^2-2\F(\Jimm(r)+2)^2-2\F(\Jimm(r)+1)^2-\F(q)^2=0.
$$
Noting that 
$$
\Jimm\left(\frac{r+1}{r}\right)=\Jimm(r)+1, \quad 
\Jimm\left(\frac{2r+1}{r+1}\right)=\Jimm(r)+2, \quad 
\Jimm\left(\frac{3r+2}{2r+1}\right)=\Jimm(r)+3
$$ 
we get
$$
\F\left(\Jimm\left(\frac{3r+2}{2r+1}\right)\right)^2-
2\F\left(\Jimm\left(\frac{2r+1}{r+1}\right)\right)^2-
2\F\left(\Jimm\left(\frac{r+1}{r}\right)\right)^2-
\F(\Jimm(r))
$$
Since $\F(\Jimm(q))=\num(q)$, 
$$
\num\left(\frac{3r+2}{2r+1}\right)^2-
2\num\left(\frac{2r+1}{r+1}\right)^2-
2\num\left(\frac{r+1}{r}\right)^2-
\num(r)^2
$$
writing $r=m/n$ as a reduced fraction, this is equivalent to
$$
(3m+2n)^2-2(2m+n)^2-2(m+n)^2-m^2,
$$
and it can be checked that this last expression is indeed vanishing.
\subsection{Outer automorphism of $\pgl$ in terms of the conumerator}
$\pgl$ is generated by the M\"obius transformations
$V:x\to -x$, $K: x\to 1-x$ and $U: x\to 1/x$. Denote by $\alpha$ Dyer's \cite{Dyer} outer automorphism sending 
$V\to UV$, $K\to K$, $U \to U$.  
Since  ${\mathrm{Out}}(\pgl)\simeq \Z/2\Z$,  $\alpha$ is involutive. 
It was proved in \cite{Uludag} that  $\Jimm$ is induced by $\alpha$ in the sense that 
\begin{align}\label{conjug}
\Jimm (M (\Jimm (x) )= (\alpha M)(x) \quad (x\in \R\setminus \Q, M\in \pgl)
\end{align}
Inserting $\Jimm(x)$ in place of $x$ and using the involutivity of Jimm  gives
$\Jimm(M(x))=(\alpha M)(\Jimm(x))$.
Denote the image of a matrix 
$\left(\begin{array}{cc} p&q\\r&s \end{array}\right)\in \mathrm{GL}_2(\Z)$ 
in $\pgl$ by 
$\left[\begin{array}{cc} p&q\\r&s \end{array}\right]$.
It is possible to express $\alpha$  in terms of the conumerator, as follows:
\begin{lemma}\label{magic}
Suppose $p,q,r,s\in \Z_{>0}$ and $M=\left[\begin{array}{cc} p&q\\r&s \end{array}\right]\in \pgl$.
Then
\begin{align}\label{mmm}
 \alpha(M)=
\left[\begin{array}{ccc}  \con\frac{2p+q}{2r+s}-\con\frac{p+q}{r+s}
&\quad &\con\frac{2p+q}{2r+s}-2\con\frac{p+q}{r+s}\\&&\\\con\frac{2r+s}{2p+q}-\con\frac{r+s}{p+q}&\quad&\con\frac{2r+s}{2p+q}-2\con\frac{r+s}{p+q}\end{array}\right]
\end{align}\end{lemma}
\begin{proof}
Set 
$$
M= \left[\begin{array}{cc} p&q\\r&s \end{array}\right]\in \pgl, \quad
\alpha M= \left[\begin{array}{cc} a&b\\c&d \end{array}\right]\in \pgl.
$$
Then $
\Jimm (M (\Jimm (1) ) = (\Jimm M)(1) \implies \Jimm ({(p+q)}/{(r+s)})={a+b}/{c+d}$, \\
$\Jimm (M (\Jimm (2) ) = (\Jimm M)(2) \implies \Jimm ({(2p+q)}/{(2r+s)})={2a+b}/{2c+d}$.\\
Since the fractions above are reduced, one has
\begin{eqnarray*}
\con({(p+q)}/{(r+s)})=a+b, \quad \con\left({(r+s)}/{(p+q)}\right)=c+d\\
\con\left({(2p+q)}/{(2r+s)}\right)=2a+b, \quad \con\left({(2r+s)}/{(2p+q)}\right)=2c+d
\end{eqnarray*}
The solution of this system yields the matrix (\ref{mmm}). $\qed$
\end{proof}
%
%

\subsection{Analytic extensions of the conumerator}\label{anal}
A meromorphic function satisfying the functional equations (\ref{nmr111}-\ref{nmr222}) for the numerator, must coincide on $\Q_{>0}$ with the numerator up to a constant factor; hence, such a function cannot exist. Idem for the conumerator function. 

If we look for a function analytic on the upper half plane $\mathcal H$ and satisfying (\ref{nmr111}-\ref{nmr222}), we are faced with the fact that   (\ref{nmr111}) does not make sense in $\mathcal H$. The same is true for the system (\ref{cnm1}-\ref{cnm2}). On the other hand, in the case of the conumerator, we may iterate (\ref{cnm1}-\ref{cnm2})  once to get a weaker system which does make sense in $\mathcal H$:
\begin{eqnarray}
f(2+z)=f(z+1)+f(1/(z+1))\implies f(2+z)=f(1+z)+f(z)\label{fournier}\\
f\left(\frac{1}{1+\frac{1}{1+z}}\right)=f(z) \implies f\left(\frac{1+z}{2+z}\right)=f(z)\label{fourier}
\end{eqnarray}
Passing to a strip model $z\to \tau$, one may express  (\ref{fourier}) as $f(\tau+\alpha)=f(\tau)$ for a certain $\alpha$, yielding a Fourier expansion for $f$. However, we have failed to produce an analytic function on $\mathcal H$ satisfying both of (\ref{fournier}-\ref{fourier}) this way. On the other hand, one can produce entire functions satisfying (\ref{fournier}) as follows. Let $g(z)$ be a function decaying rapidly as $z\to \infty$ along the lines parallel to the $x$-axis, so that the ``convolution'' below converges on the complex plane:
\begin{align*}
f(z):=&\sum_\Z F_n g(n-z)\implies\\
f(z+1)=&\sum_\Z F_n g(n-z-1)=\sum_\Z F_{n+1} g(n-z)\implies\\ 
f(z)+f(z+1)=&\sum_\Z F_n g(n-z)+\sum_\Z F_{n+1} g(n-z)=f(z+2).
\end{align*}
Since  $F_n$ grows exponentially, the first candidate for $g$ is the gaussian $g(z)=e^{-\alpha z^2}$ with $\alpha>0$.
One has then
$$
f(z)=\sum_\Z F_n e^{-\alpha (z-n)^2}=\sum_\Z \frac{\varphi^n-(-\varphi)^{-n}}{\sqrt5} e^{-\alpha (z-n)^2}=
$$
$$
=
\frac1{\sqrt5}\sum_\Z (e^{n \log \varphi}-(-1)^ne^{-n \log \varphi}) e^{-\alpha (z-n)^2}
$$
$$
=
\frac{e^{-\alpha z^2}}{\sqrt5}
\left(
\sum_\Z  e^{-\alpha n^2 +(2\alpha z+\log\varphi)n}
-
\sum_\Z  e^{-\alpha n^2 +(2\alpha z-\log\varphi+\pi i)n } 
\right)
$$
One has for the $\vartheta$ function
$$
\vartheta(z; \tau) = \sum_\Z e^{\pi i n^2 \tau + 2 \pi i n z}
\implies
\vartheta(iz; i\tau) = \sum_\Z e^{-\pi  n^2 \tau - 2 \pi  n z}
$$
$$
\vartheta(i\frac{2\alpha z+\log\varphi}{-2\pi}; \frac{\alpha i}{\pi}) = \sum_\Z  e^{-\alpha n^2 +(2\alpha z+\log\varphi)n}
$$
$$
\vartheta(i\frac{2\alpha z+\log\varphi+\pi i}{-2\pi}; \frac{\alpha i}{\pi}) = \sum_\Z  e^{-\alpha n^2 +(2\alpha z+\log\varphi+\pi i)n}
$$
$$
\implies f(z)=\frac{e^{-\alpha z^2}}{\sqrt5}
\left(\vartheta(i\frac{2\alpha z+\log\varphi}{-2\pi}; \frac{\alpha i}{\pi})-\vartheta(i\frac{2\alpha z-\log\varphi+\pi i}{-2\pi}; \frac{\alpha i}{\pi})\right)
$$
The simplest case of which is $\alpha=\pi$, $\implies$
$$
f(z)=\frac{e^{-\pi z^2}}{\sqrt5}
\left[\vartheta\left((\frac{\log\varphi}{2\pi}-z)i; i\right)-\vartheta\left((\frac{\log\varphi+\pi i}{2\pi}-z)i; i\right)\right]
$$
This idea of finding analytic solutions to the Fibonacci recurrence equation can be applied to linear recurrence equations of arbitrary order and with constant coefficients, compare with \cite{Charm} and \cite{Watson}.

\subsection{Analytic extensions of  Jimm}\label{anali}
Concerning the question of analytic `{extension}' of $\Jimm$ to $\mathcal H$, we have the following observation.
A function $f:D\to \C$ is said to be {\it equivariant} with respect to a $\psl$-action on $D$ if 
$f(Mz)=Mf(z)$ holds for every $M\in \psl$ and $z\in D$. It is known that $\psl$-equivariant functions analytic on $\mathcal H$ exists and the Schwartizan derivative of an equivariant form 
is a weight-4 modular form, see~\cite{sebbar}.
On the other hand, $\Jimm:\R\setminus \Q\to \R$ is what is called an  $\alpha$-equivariant function, in the sense that it satisfies $\Jimm(Mx)= \alpha(M) \Jimm(x)$ for every $M\in \pgl$ and $x\in \R\setminus \Q$. It is natural to ask whether there exists a function analytic on $\mathcal H$ satisfying this equivariance property. 
The subgroup $\psl$ of $\pgl$ is generated by  $S:=UV:x\to -1/x$ and $L:=KU:x\to1-1/x$, and the functional equations for $\Jimm$ for $\psl$ reads
\begin{eqnarray}
f(Sx)=\alpha(S)f(z)\implies f(-1/z)=-f(z)\label{e1}\\
f(Lx)=\alpha(L)f(z)\implies f(1-1/z)=1-1/f(z)\label{e2}
\end{eqnarray}
and combining these two equations gives
\begin{eqnarray}\label{combined}
f(1+z)=1+1/f(z) \implies f(n+z)=\frac{F_{n+1}f(z)+F_{n}}{F_{n}f(z)+F_{n-1}}
\end{eqnarray}
Setting $z=i$, Equation~\ref{e1} yields $f(i)=0$ and the same substitution in (\ref{combined}) gives $f(1+i)=\infty$ and $f(n+i)=F_{n}/F_{n-1}$. One can similarly compute the image under $f$ of the points belonging to the $\psl$-orbit of $(-1+\sqrt{3})/2$.
\begin{theorem}
There exists a meromorphic equivariant form $f$ satisfying (\ref{e1}-\ref{e2}). 
\end{theorem}
\begin{proof}
This follows from the general results of \cite{yoshida} (see also ~\cite{sebbar}, \cite{sebbar2}). It is possible to write $f$ down quite explicitly in terms of hypergeometric functions, but since it is quite involved, we don't reproduce it here.$\qed$
\end{proof}
Note that $f$ is not unique.
We don't know whether $\Jimm$ is the boundary limit of $f$ in some sense. 

\sherh{We consider the representation of the free group 
$$
\rho: \langle u,v  \rangle \to \mathsf{GL}(\Z),
\quad u\to \left(\begin{matrix}1&0\\0&-1\end{matrix}\right),
\quad v\to \left(\begin{matrix}1&-1\\1&0\end{matrix}\right),
\quad uv\to \left(\begin{matrix}1&-1\\-1&0\end{matrix}\right),
$$
For the eigenvalues one has
$$
\rho(u): \lambda_{1,2}=\pm 1, \quad \rho(v): \mu_{1,2}=\omega, -1/\omega, \quad \rho(uv):\nu_{1,2}=\bar\varphi,-1/\bar\varphi
$$
where $\omega=\frac{1+i\sqrt{3}}{2}$ and $\bar\varphi=\frac{1+\sqrt{5}}{2}$. For the Riemann scheme
$$
\left(\begin{matrix}0&1&\infty\\\rho_1&\sigma_1&\tau_1\\\rho_2&\sigma_2&\tau_2\end{matrix}\right)
$$
we have 
$$
\lambda_{1,2}=\exp{2\pi i \rho_1}, \exp{2\pi i \rho_2},
\quad 
\mu_{1,2}=\exp{2\pi i \sigma_1}, \exp{2\pi i \sigma_2},
\quad
\nu_{1,2}=\exp{2\pi i \tau_1}, \exp{2\pi i \tau_2}.
$$
Thus, up to an integer and permutations we have
$$
\rho_1=1, \rho_2=\frac12, \quad 
\sigma_1=\frac16, \sigma_2=-\frac16, \quad
\tau_1=\frac{\log \bar\varphi}{2\pi i}, \tau_2=-\frac{\log \bar\varphi}{2\pi i}+\frac12
$$
$$
P\left(\begin{matrix}
0&1&\infty\\
1&\frac16&\frac{\log \bar\varphi}{2\pi i}\\
\frac12&-\frac16&\frac12-\frac{\log \bar\varphi}{2\pi i}
\end{matrix}\right)(x)=
x(x-1)^\frac16
P\left(\begin{matrix}
0&1&\infty\\
0&0&\frac76+\frac{\log \bar\varphi}{2\pi i}\\
-\frac12&-\frac13&\frac53-\frac{\log \bar\varphi}{2\pi i}
\end{matrix}\right)(x)
$$
A bit simplified
$$
x(x-1)^\frac16
P\left(\begin{matrix}
0&1&\infty\\
0&0&\frac16+\frac{\log \bar\varphi}{2\pi i}\\
\frac12&\frac23&\frac23-\frac{\log \bar\varphi}{2\pi i}
\end{matrix}\right)(x) \to 
P\left(\begin{matrix}
0&1&\infty\\
0&0&\alpha\\
1-\gamma&\gamma-\alpha-\beta&\beta
\end{matrix}\right)
$$
Hence $\alpha=\frac16+\frac{\log \bar\varphi}{2\pi i}$, $\beta=\frac23-\frac{\log \bar\varphi}{2\pi i}$, $\gamma=\frac12$.
The corresponding Gauss hypergeometric equation is
$$
x(x-1)\frac{d^2u}{dx^2}+\bigl\{\gamma-(\alpha+\beta+1)x\bigr\}\frac{du}{dx}-\alpha\beta u=0
$$
$$
x(x-1)\frac{d^2u}{dx^2}+\left(\frac12-\frac{11}{6}x\right)\frac{du}{dx}-\alpha\beta u=0
$$
which admits the Frobenius solution
$$
F(\alpha, \beta, \gamma, x)=\sum_{m=0}^\infty \frac{(\alpha)_m(\beta)_m}{(\gamma)_m(1)_m}x^m
$$
}

\sherhh{
$$
\cds'(x):=2\con(x)\con'(x)-\con'(x+1)\con(x-1)-\con(x+1)\con'(x-1)
$$
$$
\cds'(x):=2\con(x)\con'(x)-\con'(x+1)\con(x-1)-\con(x+1)\con'(x-1)
$$
$$
=2\con^2(x)\num(1/x)^2-\con(x+1)\con(x-1)\num(\frac{1}{x+1})^2-\con(x+1)\con(x-1)\num(\frac{1}{x-1})^2
$$
$$
=2\con^2(x)\num(1/x)^2-\con(x+1)\con(x-1)(\num(\frac{1}{x+1})^2+\num(\frac{1}{x-1}))^2
$$
$$
=2(\con^2(x)-\con(x+1)\con(x-1))\num(1/x)^2=2\cds(x)\num(1/x)^2
$$
$$
\implies \cds'(x)=2\cds(x)\num(1/x)^2 \implies \cds(x)=\exp(2\int \num(1/x)^2)
$$
$$
\implies \cds'(1+x)=2\cds(x+1)\num(1/(x+1))^2=-2\cds(x)\num(1/x)^2=-\cds'(x)
$$
$$
\implies \cds'(1+x)=-\cds'(x)
$$
Can we get new Fibonacci identities this way?
}

\begin{acknowledgements}
This work is supported by the T\"{U}B\.{I}TAK grant 115F412 and the Galatasaray University research grant ****.
We are grateful to Masaaki Yoshida for some fruitful discussions.  We are particularly grateful to the editor of {\it Fibonacci Quarterly} for rejecting the paper and gaining us invaluable time to improve it considerably.

\end{acknowledgements}



\vfill

\newpage
\renewcommand{\arraystretch}{1}

\noindent {\bf Appendix I} $\con(k/41)$, $k=1,2,..200$.

\bigskip

{\small
$$
\begin{array}{|l|l|}\hline
1& 59369\!\times\!2789\\ \hline
2& 7\!\times\!2161\\ \hline 
3& 5\!\times\!7\!\times\!19\\ \hline 
4& 5\!\times\!67\\ \hline 
5& 11\!\times\!19\\ \hline 
6& 3\!\times\!5\!\times\!7\\ \hline 
7& 103\\ \hline 8& 
3^2\!\times\!23\\ \hline 
\textcolor{green}{9}& \textcolor{green}{31}\\ \hline 
10& 7^3\\ \hline
\bf{11}& \bf{17}\\ \hline 
12& 2^3\!\times\!3\\ \hline 
13& 5\!\times\!13\\ \hline 
14& 3\!\times\!281\\ \hline 
15& 23\\ \hline 
16& 19\\ \hline 
17& 29\\ \hline 
\bf{18}& \bf{17}\\ \hline 
19& 3^4\\\hline 
20& 89\!\times\!199\\ \hline 
\textcolor{red}{21}& \textcolor{red}{3\!\times\!5\!\times\!11\!\times\!41}\\ \hline 
\bf{22}& \textcolor{green}{31}\\ \hline 
23& 2\!\times\!3\\ \hline 
24& 2^2\!\times\!3\\ \hline 
25& 2^2\\\hline
26& 7\\ \hline 
27& 233\\ \hline 
28& 47\\ \hline 
\bf{29}& \bf{17}\\ \hline 
30& 13\\ \hline 
31& 199\\ \hline
\bf{32}& \bf{17}\\ \hline 
33& 131\\ \hline 	
34& 2\!\times\!3\!\times\!11\\ \hline 
35& 2^6\\ \hline
36& 3\!\times\!43\\ \hline 
37& 3^2\!\times\!23\\ \hline 
38& 3\!\times\!137\\ \hline 
39& 9349\\ \hline 
\textcolor{red}{40}& 
{\mbox{\tiny  \textcolor{red}{$3\!\times\!5\!\times\!7\!\times\!11\!\times\!41\!\times\!2161$}}}\\ \hline 
41& 1 \\ \hline 
42& 433494437\\ \hline 
43& 3\!\times\!43\!\times\!307\\ \hline 
44& 1741\\ \hline 
45& 877\\ \hline
46& 547\\ \hline 47& 2\!\times\!137\\ \hline 48& 2^4\!\times\!17\\ 
\hline 
49& 5
\!\times\!109\\ \hline 50& 79\\ \hline 
\end{array}
\begin{array}{|l|l|}\hline
51& 3\!\times\!5\!\times\!59\\ \hline 52& 47\\ \hline 53&
5\!\times\!13\\ \hline 54& 3\!\times\!59\\ \hline 55& 19\!\times\!101\\ \hline 56& 53\\ \hline 
57& 2\!\times\!3\!\times\!7\\ \hline 58& 2\!\times\!5\!\times\!7\\ \hline 59& 2^3\!\times\!5\\ \hline 
60& 193\\ \hline 61& 42187\\ \hline 62& 7\!\times\!4463\\ \hline 
63& 11\!\times\!13\\\hline 
 64& 29\\ \hline 65& 53\\ \hline 66& 3^3\\ \hline 67& 37\\ \hline 68& 
7\!\times\!11\!\times\!17\\ \hline 69& 3\!\times\!53\\ \hline 70& 2\!\times\!29\\ \hline 71& 43\\ \hline
72& 3\!\times\!13\!\times\!19\\ \hline 73& 5\!\times\!13\\ \hline 74& 7\!\times\!67\\ \hline 75
& 5\!\times\!47\\ \hline 76& 233\\ \hline 77& 467\\ \hline 78& 7\!\times\!107\\ \hline 79
& 1487\\ \hline 
\textcolor{red}{80}& \textcolor{red}{3\!\times\!5^2\!\times\!11\!\times\!41}\\ \hline 81& 370248451\\ \hline 82
& 2\\ \hline 83& 
\mbox{{\tiny{\tiny $2\!\times\!3^2\!\times\!83\!\times\!281\!\times\!1427$}}}\\ \hline 84& 2\!\times\!5\!\times\!
13\!\times\!421\\ \hline 85& 2\!\times\!3\!\times\!401\\ \hline 86& 2^2\!\times\!3\!\times\!101\\ \hline 
87& 2^2\!\times\!3^3\!\times\!7\\ \hline 88& 379\\ \hline 89& 3\!\times\!5^3\\ \hline 90& 
2^4\!\times\!47\\ \hline 91& 2\!\times\!5\!\times\!11\\ \hline 92& 2^2\!\times\!307\\ \hline 93& 
2^6\\ \hline 94& 89\\ \hline 95& 2\!\times\!11^2\\ \hline 96& 2\!\times\!1381\\\hline 97
& 2^2\!\times\!19\\ \hline 98& 61\\ \hline 99& 3^2\!\times\!11\\ \hline 100& 3\!\times\!
19\\ \hline
\end{array}
\begin{array}{|l|l|}
 \hline  101& 2\!\times\!137\\ \hline 102& 2\!\times\!3\!\times\!67\!\times\!149\\ \hline 103& 
2\!\times\!31\!\times\!613\\ \hline 104& 2\!\times\!3\!\times\!29\\ \hline 105& 5\!\times\!7\\ \hline 106&
5\!\times\!13\\ \hline 
\bf{107}& \textcolor{green}{31}\\ \hline 108& 2^2\!\times\!11\\ \hline 109& 2\!\times\!3\!\times\!
257\\ \hline 110& 2\!\times\!103\\ \hline 111& 3\!\times\!5^2\\\hline 112& 2^3\!\times\!7
\\ \hline 113& 2^2\!\times\!5\!\times\!47\\ \hline 
\textcolor{red}{114}& \textcolor{red}{2\!\times\!41}\\ \hline 115& 2^3\!\times\!3
\!\times\!5^2\\ 116& 7\!\times\!43\\ \hline 117& 3^3\!\times\!11\\ \hline 118& 2^2\!\times\!
149\\ \hline 119& 2^2\!\times\!239\\ \hline 120& 2\!\times\!13\!\times\!73\\ \hline 121& 2\!\times\!
21587\\ \hline 122& 2\!\times\!1109\!\times\!213067\\ \hline 123& 3\\ \hline 124& 5\!\times\!
23\!\times\!229\!\times\!39209\\ \hline 125& 17\!\times\!31\!\times\!179\\ \hline 126& 11\!\times\!
13\!\times\!29\\ \hline 127& 2089\\ \hline 128& 1303\\ \hline 129& 653\\ \hline 130& 
647\\ \hline 131& 1297\\ \hline 132& 3^3\!\times\!7\\ \hline 133& 2113\\ \hline 134
& 3\!\times\!37\\ \hline 135& 2\!\times\!7\!\times\!11\\ \hline 136& 419\\ \hline 137& 31
\!\times\!151\\ \hline 138& 3\!\times\!43\\ \hline 139& 103\\ \hline 140& 13^2\\ \hline 141&
97\\ \hline 142& 467\\ \hline 143& 5\!\times\!17\!\times\!1201\\ \hline 144& 69247\\ \hline
145& 317\\ \hline 146& 2^6\\ \hline
\textcolor{magenta}{147}& \textcolor{magenta}{2\!\times\!59}\\ \hline \textcolor{magenta}{148}& \textcolor{magenta}{2\!\times\!29}
\\ \hline 149& 3^4\\ \hline150& 2851\\ \hline
\end{array}
\begin{array}{|l|l|}\hline  151& 5\!\times\!73\\ \hline 152& 7\!\times\!
19\\ \hline 153& 3^2\!\times\!11\\ \hline 
\textcolor{red}{154}& \textcolor{red}{41^2}\\ \hline155& 3\!\times\!7^2\\ \hline 
156& 1069\\ \hline 157& 2^3\!\times\!67\\ \hline 158& 2\!\times\!5\!\times\!53\\ \hline 159
& 1063\\ \hline 160& 5\!\times\!11\!\times\!31\\ \hline 161& 5\!\times\!677\\ \hline 162& 
13\!\times\!5923\\ \hline 163& 72043\!\times\!11699\\ \hline 164& 5\\ \hline 165& 
1631643593\\ \hline 166& 23\!\times\!6481\\ \hline 167& 6553\\ \hline 168& 3301\\ \hline
169& 29\!\times\!71\\ \hline 170& 2^3\!\times\!3\!\times\!43\\ \hline 171& 2\!\times\!7\!\times\!
73\\ \hline 172& 3\!\times\!683\\ \hline 173& 13\!\times\!23\\ \hline 174& 13\!\times\!257\\ \hline
175& 5^2\!\times\!7\\ \hline 176& 3^5 \\ \hline177& 661\\ \hline 178& 3^2\!\times\!
827\\ \hline 
\textcolor{red}{179}& \textcolor{red}{5\!\times\!41}\\ \hline 
\textcolor{red}{180}& \textcolor{red}{2^2\!\times\!41}\\ \hline 181& 2^2\!\times\!67
\\ \hline 182& 2\!\times\!7\!\times\!11\\ \hline 183& 3\!\times\!13\!\times\!19\\ \hline 184& 
161983\\ \hline 185& 3^2\!\times\!17\!\times\!701\\ \hline 186& 491\\ \hline 187& 3^2\!\times\!
11\\ \hline 188& 3\!\times\!61\\ \hline 189& 89\\ \hline 190& 5^3\\ \hline 191& 
23\!\times\!191\\ \hline 192& 571\\ \hline 193& 2^4\!\times\!13\\ \hline 194& 5\!\times\!31
\\ \hline 195& 2621\\ \hline 196& 229\\ \hline 197& 1669\\ \hline 198& 3^3\!\times\!
31\\ \hline 199& 827\\ \hline 200& 3\!\times\!7\!\times\!79 \\\hline
\end{array}
$$}

\end{document}